\documentclass[11pt]{amsart}

\usepackage{amsmath,amssymb,latexsym,soul,cite,mathrsfs}

\usepackage{color,enumitem,graphicx}
\usepackage[colorlinks=true,urlcolor=blue,
citecolor=red,linkcolor=blue,linktocpage,pdfpagelabels,
bookmarksnumbered,bookmarksopen]{hyperref}
\usepackage[english]{babel}

\usepackage[left=2.9cm,right=2.9cm,top=2.8cm,bottom=2.8cm]{geometry}
\usepackage[hyperpageref]{backref}

\usepackage[colorinlistoftodos]{todonotes}
\makeatletter
\providecommand\@dotsep{5}
\def\listtodoname{List of Todos}
\def\listoftodos{\@starttoc{tdo}\listtodoname}
\makeatother

\numberwithin{equation}{section}
\newtheorem{theorem}{Theorem}[section]

\newtheorem{lemma}[theorem]{Lemma}

\newtheorem{claim}[theorem]{Claim}

\newtheorem{remark}{Remark}
\newtheorem{definition}{Definition}[section]

\begin{document}

	\title[Existence of solution for a class of heat equation ...]
	{Existence of solution for a class of heat equation involving the 1-Laplacian operator }
	\author{Claudianor O. Alves and Tahir Boudjeriou}
	\address[Claudianor O. Alves ]
	{\newline\indent Unidade Acad\^emica de Matem\'atica
		\newline\indent 
		Universidade Federal de Campina Grande 
		\newline\indent
		e-mail:coalves@mat.ufcg.edu.br
		\newline\indent	
		58429-970, Campina Grande - PB, Brazil} 
	
	\address[Tahir Boudjeriou]
	{\newline\indent  
		Department of Mathematics
		\newline\indent Faculty of Exact Sciences
		\newline\indent Lab. of Applied Mathematics
		\newline\indent University of Bejaia, Bejaia, 6000, Algeria 
		\newline\indent
		e-mail:re.tahar@yahoo.com}

	\pretolerance10000
	
	\begin{abstract}
		This paper concerns the existence of global solutions for the following class of heat equation involving the 1-Laplacian operator of the Dirichlet problem 
		$$
		\left\{	
		\begin{array}{llc}
			u_{t}-\Delta_1 u=f(u) & \text{in}\ & \Omega\times (0, +\infty) , \\
			u =0 & \text{in} & \partial\Omega\times (0, +\infty), \\
			u(x,0)=u_{0}(x)& \text{in} &\Omega , 
		\end{array}\right.
		\leqno{(P)}
		$$
		where  $\Omega \subset \mathbb{R}^N$ is a smooth bounded domain, $N \geq 1$ and $f:\mathbb{R}\to \mathbb{R}$ is a continuous function satisfying some technical conditions,  and $\Delta_1 u=\text{div}\left(\frac{Du}{|Du|}\right)$ denotes the 1-Laplacian operator.  The existence of global solution is done by using an approximation technique that consists in working with a class of $p$-Laplacian problem associated with $(P)$ and then taking the limit when $ p \to 1^+$ to get our results.
	\end{abstract}
	\thanks{C.O. Alves was partially supported by CNPq/Brazil  304804/2017-7.}
	\maketitle	
	\section{Introduction and the main results} 
	
	In this paper, we are concerned with the existence of global solutions for the following class of heat equation involving the 1-Laplacian operator of the Dirichlet problem 
	\begin{equation}\label{P1}\left\{
		\begin{array}{llc}
			u_{t}-\Delta_1 u=f(u) & \text{in}\ & \Omega\times (0, +\infty) , \\
			u =0 & \text{in} & \partial\Omega\times (0, +\infty), \\
			u(x,0)=u_{0}(x)& \text{in} &\Omega , 
		\end{array}\right.
	\end{equation}
	where  $\Omega \subset \mathbb{R}^N$ is a smooth bounded domain, $N \geq 1$ and $f:\mathbb{R}\to \mathbb{R}$ is a continuous function and $\Delta_1 u=\text{div}\left(\frac{Du}{|Du|}\right)$ denotes the 1-Laplacian operator.
	
	In recent decade, the operator $\Delta_1$ has received a special attention after the definition of solution introduced by Attouch, Buttazzo and Michaille \cite{Andreu}, where many authors have analyzed the existence of solution for  problem (\ref{P1}) in the stationary case, that is, for the following class of problems 
	\begin{equation}\label{P12}\left\{
		\begin{array}{llc}
			-\text{div}\left(\frac{Du}{|Du|}\right)=f(u) & \text{in}\ & \Omega , \\
			u =0 & \text{in} & \partial\Omega. \\
		\end{array}\right.
	\end{equation}

	Accurately, Degiovanni and Magrone \cite{DegiovanniMagrone} studied the version of Br\'ezis-Nirenberg
	problem to the (\ref{P12}), by applying Linking theorem, the functional energy has been extended in Lebesgue space $L^{1^*}(\Omega)$, in order to recover the compactness embedding. 
	Concerning the spectrum of the 1-Laplacian operator, by using the same approach, Chang \cite{Chang2} proves the existence of a sequence of eigenvalues.
	
	Different approaches have been taken to attack (\ref{P12}) under various hypotheses on the nonlinearity $f$. In \cite{FigueiredoPimenta1}, Figueiredo and Pimenta studied a related problem to (\ref{P12}), where the nonlinearity has a subcritical growth and their main results establishes the existence of a nontrivial ground-state solution.
	
	As regards quasilinear problems, depending on some features of the differential
	operator to be considered, it is worth to work with it in suitable space like the space of
	functions of bounded variation, $BV(\Omega).$ We may address the question of finding
	critical points of a functional in the space of functions of bounded variation, $BV (\Omega),$ in a setting
	in which coerciveness and smoothness are lost. In other words, the main difficulties arise mainly to the lack of smoothness on the energy functional associated to (\ref{P12}) and the lack of reflexiveness on $BV(\Omega).$  Indeed, the energy functional is not $C^1$ and we find some hindrances to show that functionals
	defined in this space satisfy compactness conditions like the Palais-Smale. Meanwhile, a lot of attention has
	been paid recently to that space, for example see \cite{AlvesPimenta,Abm,BCN,Chang2,Demengel1,DegiovanniMagrone,FigueiredoPimenta3,Kawohl,sergio,Pimenta2} and references therein, since it is the natural environment
	in which minimizers of many problems can be found, especially in problems
	involving interesting physical situations, in capillarity theory and existence of minimal surfaces and as application of variational approach to image
	restoration.
	
	However, related to the evolution problem (\ref{P1}) we have few references. In \cite{Andreu}, Andreu, Ballester, Caselles and Maz\'on studied the existence of solution for the following class of evolution equations 
	\begin{equation}\label{P13}\left\{
		\begin{array}{llc}
			u_{t}-\text{div}\left(\frac{Du}{|Du|}\right)=0 & \text{in}\ & \Omega\times (0, +\infty) , \\
			u =\varphi & \text{in} & \partial\Omega\times (0, +\infty), \\
			u(x,0)=u_{0}(x)& \text{in} &\Omega , 
		\end{array}\right.
	\end{equation}
	where $u_0 \in L^{1}(\Omega)$ and $\varphi \in L^{\infty}(\Omega)$. By using the techniques of completely accretive
	operators and the Crandall-Liggett semigroup generation theorem \cite{Crandall}, the authors were able to prove  the existence and uniqueness of entropy solution. The same problem, but involving the Neumann boundary condition, that is  $\frac{\partial u}{\partial \nu}=0$ in  $\partial\Omega\times (0, +\infty)$, it was considered in Andreu, Ballester, Caselles, and Maz\'on \cite{Andreu3}. Still related to the existence of solution to problem (\ref{P13}), we would like to cite a paper by Hardt and Zhou \cite{HardtZhou}, where the authors used 
	an approximation technique, which consists  in working with a class of nondegenerate parabolic approximation problem, and after some estimates, they were able to prove the existence of solution for the original problem. In \cite{Andreu2}, Andreu, Caselles, D\'iaz and Maz\'on studied the asymptotic profile of solutions to (\ref{P13}) near the extinction time for Dirichlet and Neumann boundary conditions. 
	
	The literature concerning the inhomogeneous case of problem (\ref{P13}) is poor and, to our best knowledge, there a few papers in which the authors studied the existence and uniqueness of solutions. In that direction we  mention a result by Le\'on and Webler \cite{SW}, on the following problem 
	\begin{equation}\label{P14}\left\{
		\begin{array}{llc}
			u_{t}-\text{div}\left(\frac{Du}{|Du|}\right)= f(x,t)& \text{in}\ & \Omega\times (0, +\infty) , \\
			u =0 & \text{in} & \partial\Omega\times (0, +\infty), \\
			u(x,0)=u_{0}(x)& \text{in} &\Omega , 
		\end{array}\right.
	\end{equation}
	where  $u_{0} \in L^{2}(\Omega)$ and $f\in L_{loc}^{1}(0, +\infty;L^{2}(\Omega) )$ . In that work the authors proved global existence and uniqueness of solutions for (\ref{P14}) via a parabolic $p$-Laplacian problem and then taking the limit $p\rightarrow 1^{+}$.  In\cite{DH1}, by means of nonlinear semigroup, Hauer and Maz\'on  studied the existence of strong  solutions to problem (\ref{P14}) with a global Lipschitz continuous function $f(x,u)$ in the second variable instead of $f(x,t)$.  Finally, for some recent results on parabolic equations involving the fractional 1-Laplacian operator, we refer the reader to \cite{DH1} and \cite{DH2}.
	
	Motivated by studied make in \cite{DH1, Andreu,Andreu2,Andreu3,HardtZhou,SW}, we intend to prove two existence results of solution for (\ref{P1}) by supposing different conditions on $f$ and $\Omega$. Our first result is devoted to the radial case, where we assume that $\Omega$ is an annulus  region of the form  
	\begin{equation} \label{omega}
		\Omega=\{x\in \mathbb{R}^N\;:\;a<|x|<b\},
	\end{equation}
	where $0<a<b<+\infty$, while $f:\mathbb{R}\to \mathbb{R}$ is a continuous function satisfying  the following conditions:
	
	\noindent There is $p_0 \in (1,2)$ such that 
	$$
	\lim_{t \to 0} \frac{f(t)}{|t|^{p_0-1}}=0. \leqno{(f_1)}
	$$
	There is $\theta>1$ such that
	$$
	0<\theta F(t) \leq f(t)t, \quad t \in \mathbb{R}\setminus \{0\}. \leqno{(f_2)}
	$$
	
	In order to state our condition on the initial data $u_0$, we need to fix some notations. Hereafter, let us denote by $W_{0,rad}^{1,p}(\Omega)$ and $L_{rad}^{2}(\Omega)$ the subspaces of $W^{1,p}_{0}(\Omega)$ and $L^{2}(\Omega)$ respectively that are formed by radial functions.  It is well known that the embedding below
	\begin{equation} \label{EMBW1p}
		W_{0,rad}^{1,p}(\Omega) \hookrightarrow C(\overline{\Omega})
	\end{equation}
	is compact, whose proof is an immediate consequence of \cite[Chapitre 6, Lemme 1.1]{Kavian}. In particular, we have the compact embedding
	$$
	W_{0,rad}^{1,p}(\Omega) \hookrightarrow L^{q}(\Omega), \quad \forall q \in [1,\infty].
	$$
	The above embedding permits to consider the functional $E:W^{1,1}_{0,rad}(\Omega) \to \mathbb{R}$ given by 
	$$
	E(u)=\int_{\Omega}|\nabla u|\,dx-\int_{\Omega}F(u)\,dx.
	$$
	Associated with $E$ we have the Nehari set defined by 
	$$
	\mathcal{N}_{rad}=\left\{u \in W^{1,1}_{0,rad}(\Omega) \setminus \{0\}\,:\, \int_{\Omega}|\nabla u|\,dx=\int_{\Omega}f(u)u\,dx\right\},
	$$
	and the real number
	$$
	d=\inf_{u \in \mathcal{N}_{rad}}E(u).
	$$
	The potential well associated with problem $\eqref{P1}$ is the set
	\begin{equation}\label{pt7}
		W_{rad} =\left\{u\in  W^{1,1}_{0,rad}(\Omega)\,:\, E(u)< d\; \mbox{and} \; I(u)>0 \right\}\cup \{0\},
	\end{equation}
	where $I(u)=\displaystyle \int_{\Omega}|\nabla u|\,dx-\int_{\Omega}f(u)u\,dx$ for all $u \in W^{1,1}_{0,rad}(\Omega).$
	
	In the sequel, we give the following two definitions.
	\begin{definition} (radial solution)\label{Deff1} A function $u\in L^{\infty}(0, T; BV_{rad}(\Omega))$ will be called a radial weak solution of (\ref{P1}) if $u_{t}\in L^{2}(0, T;L_{rad}^{2}(\Omega))$,  and there exists a vector field $z\in L^{\infty}(\Omega \times(0,T), \mathbb{R}^{N})$ with $\text{div}\, z(t) \in L^{2}(\Omega)$ and $\|z(t)\|_{\infty}\leq 1$ such that
		\begin{enumerate}
			\item $u_{t}(t)-\text{div}(z(t))=f(u(t)), \;\;\text{in}\;\mathcal{D}'(\Omega) \;\;\text{a.e.}\;t\in [0,T],$
			\item $ \int_{\Omega} (z(t), Du(t))	=\int_{\Omega}|Du(t)|,$
			\item $[z(t), \nu]\in \text{sign}(-u(t))\;\;\mathcal{H}^{N-1}-\text{a.e.}\;\;\text{on}\;\partial \Omega.$
		\end{enumerate}
	\end{definition}
	\begin{definition} \label{Def1} (Maximal existence time)\label{de21} Let $u(t)$ be a solution of Problem $\eqref{P1}$. We define the maximal existence time $T_{\max}$ of $u$ as follows :
		$$ T_{\max}=\sup\{t>0\, : u=u(t)\;\text{  exists on }\, [0, T]\}.$$

		\begin{enumerate}
			\item If $T_{\max}<\infty$ we say that the solution of (\ref{P1}) blows up and $T_{\max}$ is the blow up time.
			\item If $T_{\max}=\infty$, we say that the solution is global.
		\end{enumerate}
	\end{definition}
	Now we are in position to state the first result.
	\begin{theorem}\label{TH1} Assume (\ref{omega}), $(f_1)-(f_2)$ and $u_{0} \in W_{0,rad}^{1,p_{0}}(\Omega) \cap W_{rad}$, where $p_0$ was given in $(f_1)$. Then there exists a global weak radial solution to problem (\ref{P1}). Furthermore, there holds 
		\begin{equation}\label{ENRR1}
			\int_{0}^{t}\|u_{s}(s)\|^{2}_{2}\,ds+E(u(t))\leq E(u_{0})\;\;\text{a.e.}\; t\in [0, +\infty).
		\end{equation}
	\end{theorem}

	\begin{remark} \label{R1}
		In whole this paper, $BV_{rad}(\Omega)$ denotes the subspaces of $BV(\Omega)$ that is formed by radial functions. It was shown in  Lemma \ref{LeM}, see  Section 3, that the following continuous embedding holds
		$$
		BV_{rad}(\Omega) \hookrightarrow L^{p}(\Omega) \quad \mbox{for} \quad p \in [1,\infty).
		$$
		Hence, by \cite[Corollary 4, p. 85]{JS}, $u\in C([0, T], L^{p}(\Omega) )$ for all $p \in [1,\infty)$. Then, the initial condition $u(0)=u_0$ in (\ref{P1}) exists and makes sense. 
	\end{remark}

	Concerning the non-radial case, that is, the case where $\Omega \subset \mathbb{R}^{N}(N \geq 1)$ is a bounded set with Lipschitz boundary, beside the conditions $(f_{1})-(f_2)$, we assume the following condition on the function $f$ :\\
	\noindent $(f_3)$ There exist $q\in (1, 1^{*}) $ and $C> 0$ 
	$$ 
	|f(s)|\leq C(1+|s|^{q-1}),\;\;\forall s\in \mathbb{R},
	$$
	where $1^*=\frac{N}{N-1}$ if $N \geq 2$ and $1^*=+\infty$ when $N=1$.

	Here, the Nehari set associated with $E$ is given by
	$$
	\mathcal{N}=\left\{u \in W^{1,1}_{0}(\Omega) \setminus \{0\}\,:\,  \int_{\Omega}|\nabla u|\,dx=\int_{\Omega}f(u)u\,dx\right\},
	$$
	and the potential well associated with problem (\ref{P1}) is the set
	\begin{equation}\label{pt77}
		W =\left\{u\in  W^{1,1}_{0}(\Omega)\,:\, E(u)< d \; \mbox{and} \;I(u)>0 \right\}\cup \{0\},
	\end{equation}
	where $I(u)= \displaystyle \int_{\Omega}|\nabla u|\,dx-\int_{\Omega}f(u)u\,dx$. 
	
	The second result reads as follows :
	
	\begin{theorem}\label{TH2}
		Let $\Omega \subset \mathbb{R}^{N}(N \geq 1)$ be a bounded set with Lipschitz boundary and assume that the assumptions $(f_{1})-(f_{3})$ hold and $u_0 \in W_0^{1,p_0}(\Omega) \cap W$, where $p_0$ was given in $(f_1)$.  Then, there exists a global weak  solution  to problem (\ref{P1}). Moreover, there holds 
		\begin{equation}\label{ENRR2}
			\int_{0}^{t}\|u_{s}(s)\|^{2}_{2}\,ds+E(u(t))\leq E(u_{0})\;\;\text{a.e.}\; t\in [0, +\infty).
		\end{equation}
	\end{theorem}

	Related to the Theorem \ref{TH2}, we are using the following definition of solution : 
	\begin{definition}\label{Def2}A function $u\in L^{\infty}(0, T; BV(\Omega)\cap L^{2}(\Omega))$ will be called a weak solution of (\ref{P1}) if $u_{t}\in L^{2}(0, T;L^{2}(\Omega))$, and there exist a vector field $z\in L^{\infty}( \Omega \times (0,T), \mathbb{R}^{N})$ with $\text{div}\, z(t) \in L^{2}(\Omega)$ and $\|z(t)\|_{\infty}\leq 1$  such that
		\begin{enumerate}
			\item $u_{t}(t)-\text{div}(z(t))=f(u(t)), \;\;\text{in}\;\mathcal{D}'(\Omega) \;\;\text{a.e.}\;t\in [0,T],$
			\item $ \int_{\Omega} (z(t), Du(t))	=\int_{\Omega}|Du(t)|,$
			\item $[z(t), \nu]\in \text{sign}(-u(t))\;\;\mathcal{H}^{N-1}-\text{a.e.}\;\;\text{on}\;\partial \Omega.$
		\end{enumerate}
	\end{definition}
	
	\begin{remark} \label{R2}
		In view of \cite[Lemma 1.2]{L}  and the regularity of the solution $u$ stated in Definition \ref{Def1}, we have  $u\in C([0, T], L^{q}(\Omega) )$ for each $q\in [1,2]$. Thus, the initial condition $u(0)=u_0$ exists and makes sense. 
	\end{remark}
	\subsection{Our approach} In the proof of Theorems \ref{TH1} and \ref{TH2}, we used an approximation technique that consists in working with a class of $p$-Laplacian problem associated with (\ref{P1}) and then taking the limit when $ p \to 1^+$ to get our results, more precisely, employing the potential well theory combined with Galerkin methods, we prove the existence of a global solution for the following class of quasilinear heat equation
	\begin{equation*}\label{Pp}\left\{
		\begin{array}{llc}
			u_{t}-\text{div}\left(|\nabla u|^{p-2}\nabla u\right)=f(u) & \text{in}\ & \Omega\times (0, +\infty) , \\
			u =0 & \text{in} & \partial\Omega\times (0,+\infty), \\
			u(x,0)=u_{0}(x)& \text{in} &\Omega , 
		\end{array}\right.
	\end{equation*}
	for all $p>1$, which is denoted by $u_p$. After that, we consider a sequence $p_m \to 1^+$ and show that the sequence $(u_m)$ converges for a solution of (\ref{P1}) in the sense of Definitions \ref{Def1} and \ref{Def2} respectively. In the Theorem \ref{TH1}, the reader is invited to see that we do not assume on function $f$ any growth condition from above at infinite, because in this case the domain $\Omega$ is radial and the properties of the spaces $W_{0,rad}^{1,p}(\Omega)$ and  $BV_{rad}(\Omega)$ apply an important role in our approach. However, in the proof of Theorem \ref{TH2}, we assumed that $f$ has a subcritical growth, because $\Omega$ is a general bounded. Finally, we would like to point out that we will work only with the case $N \geq 2$, because the case $N=1$ follows with few modifications. 
	
	The approximation technique by using $p$-Laplacian problems is well known to get a solution for problems involving the 1-Laplacian operator for stationary case, see for example Alves \cite{Alves0}, Demengen \cite{Demengel1,Demengel2}, Salas and Segura de Le\'on \cite{sergio}, Mercaldo, Rossi, Segura de Le\'on and Trombetti \cite{MRST},  Mercaldo, Segura de Le\'on and Trombetti \cite{MST},  Figueiredo and Pimenta  \cite{FigueiredoPimenta4}. Related to evolution case we only know the paper by Le\'on and Webler \cite{SW}. However, up to our knowledge, this is the first time that this approach is used to prove the existence of a global solution for a heat equation involving the 1-Laplacian operator when the nonlinearity $f$ is the form $f(u)$, that is, when $f$ can be a nonlinear function in the variable $u$. For example, the Theorem \ref{TH1} can be used for the nonlinearity $f(u)=|u|^{q-2}ue^{\alpha |u|^2}$ for $q>1$ and $\alpha>0$, while in Theorem \ref{TH2} we can work with $f(u)=|u|^{q-2}u+|u|^{s-2}u$ with $q,s \in (1,1^*)$.
	
	\subsection{Organization of the article}This article is organized as follows: In Section 2, we recall some notations and results involving the $BV(\Omega)$ space. In Sections 3 and 4 , we prove Theorems \ref{TH1} and \ref{TH2} respectively.
	
	\subsection{Notations} Throughout this paper, the letters $c$, $c_{i}$, $C$, $C_{i}$, $i=1, 2, \ldots, $ denote positive constants which vary from line to line, but are independent of terms that take part in any limit process. Furthermore, we denote the norm of $L^{p}(\Omega)$ for any $p\geq 1$ by $\|\,.\,\|_{p}$. In some places we will use $"\rightarrow"$,   $"\rightharpoonup"$ and $"\stackrel{*}{\rightharpoonup}"$ to denote the strong convergence, weak convergence and weak star convergence respectively.
	
	\section{Notation and preliminaries involving the space $BV(\Omega)$}
	In this section we will recall several facts on functions of bounded variation that will be used later on.
	
	Throughout the paper, without further mentioning, given an open bounded set $\Omega$ in $\mathbb{R}^{N}$ with Lipschtiz boundary, we denote by $\mathcal{H}^{N-1} $ the $(N-1)-$dimensional Hausdorff measure  and $|\Omega|$ stands for the $N$-dimensional Lebesgue measure. Moreover, we shall denote by $\mathcal{D}(\Omega)$ or $C^{\infty}_{0}(\Omega)$, the space of infinitely differentiable functions with compact support in $\Omega$ and $\nu (x) $ is the outer vector normal defined for $\mathcal{H}^{N-1}$- almost everywhere $x\in \partial \Omega$. 
	
	We will denote by $BV(\Omega)$ the space of functions of bounded variation, that is,  
	$$
	BV(\Omega)=\left\{u\in L^{1}(\Omega):\, Du \;\text{is a bounded Radon measure}\right\},
	$$
	where $Du :\Omega \rightarrow \mathbb{R}^{N}$ denotes the distributional gradient of $u$.  It can be proved that $ u\in {BV}(\Omega)$ is equivalent to $u\in L^{1}(\Omega)$ and 
	$$ \int_{\Omega}|Du|:=\text{sup}\left\{\int_{\Omega} u\,\text{div}\varphi\,dx :\, \varphi \in C^{\infty}_{0}(\Omega, \mathbb{R}^{N}),\, |\varphi (x)|\leq 1\; \forall x\in \Omega\right\}< +\infty,$$
	where $|Du|$ is the total variation of the vectorial Radon measure. It is well known that the space ${BV}(\Omega)$ endowed with the norm 
	$$
	\|u\|_{{BV}(\Omega)}:=\int_{\Omega} |Du|+\|u\|_{L^{1}(\Omega)},
	$$ 
	is a Banach space which is non reflexive and non separable. For more information on functions of bounded variation we refer the reader to \cite{Abm,Evn, Ziem}. 
	
	From \cite[Theorem 3.87]{Abm}, the notion of a trace on the boundary can be extended to functions $u\in {BV}(\Omega)$, through a bounded operator ${BV}(\Omega)\hookrightarrow L^{1}(\partial \Omega)$, which is also onto. As a consequence, an equivalent norm on ${BV}(\Omega) $ can be defined by
	\begin{equation} \label{secondnorm} 
		\|u\|:=\int_{\Omega} |Du|+\int_{\partial\Omega} |u|\, d\mathcal{H}^{N-1}.
	\end{equation} 
	In addition, by \cite[Corollary 3.49]{Abm} the following continuous embeddings hold 
	\begin{equation}
		{BV} (\Omega)\hookrightarrow L^{m}(\Omega)\;\;\text{for every}\;\;1\leq m\leq1^{*}=\frac{N}{N-1},
	\end{equation}  
	which are compact for $1\leq m<1^{*}$.

	In what follows, let us recall several important results from \cite{Anz}  which will be used throughout the paper. Following \cite{Anz}, let 
	\begin{equation}
		X(\Omega)=\left\{z\in L^{\infty}(\Omega, \mathbb{R}^{N}):\, \text{div}(z)\in L^{1}(\Omega) \right\}.
	\end{equation}
	
	If $z\in  X(\Omega)$ and $w\in {BV}(\Omega)$ we define the functional $(z, Dw): C^{\infty}_{0}(\Omega)\rightarrow \mathbb{R}$ by formula 
	\begin{equation}
		\langle (z, Dw),\varphi\rangle=-\int_{\Omega} w\varphi \text{div} (z)\,dx -\int_{\Omega}wz.\nabla \varphi\,dx, \;\;\forall \varphi \in C^{\infty}_{0}(\Omega).
	\end{equation}
	Then, by \cite[Theorem 1.5]{Anz} $(z, Dw)$ is a Radon measure in $\Omega$, 
	$$ \int_{ \Omega} (z, Dw)=\int_{ \Omega} z.\nabla w\,dx, $$
	for all $w\in W^{1,1}(\Omega)$  and 
	\begin{equation}
		\left|	\int_{\Omega} (z, Dw)\right|\leq\int_{B}|(z, Dw)|\leq \|z\|_{\infty}\int_{B} |Dw|,
	\end{equation}
	for every Borel $B$ set with $B\subseteq  \Omega $.  Moreover, besides the $BV-$norm, for any nonnegative smooth function $\varphi $  the functional given by 
	$$ w \mapsto \int_{\Omega}\varphi |Dw|, $$ 
	is lower semicontinuous with respect to the $L^{1}-$convergence, for details see \cite{Ambrosio}.

	In \cite{Anz}, a weak trace on $\partial \Omega$  of normal component of $z\in X(\Omega) $ is defined as the application $[z,\nu] : \partial \Omega \rightarrow \mathbb{R}$, such that $[z,\nu]\in L^{\infty}(\partial \Omega)$ and $\|[z,\nu]\|_{\infty}\leq \|z\|_{\infty}$. In addition, this definition coincides with the classical one, that is 
	\begin{equation}
		[z,\nu]=z.\nu, \;\text{for}\;z\in C^{1}(\overline{\Omega_{\delta}}, \mathbb{R}^{N}), 
	\end{equation} 
	where $\Omega_{\delta}=\{x\in \Omega :\, d(x, \Omega)< \delta\}$, for some $\delta > 0$ sufficiently small. We recall the Green formula involving the measure $(z,Dw)$ and the weak trace $[z,\nu]$ which was given in \cite{Anz}, namely :
	\begin{equation}\label{Gree}
		\int_{\Omega} (z, Dw)+\int_{\Omega} w\text{div} z=\int_{\partial\Omega} w[z,\nu]\, d\mathcal{H}^{N-1},
	\end{equation}
	for $z\in X(\Omega)$ and $w\in {BV} (\Omega).$

	Before concluding this section, we will prove the lemma below that is crucial in the proof of Theorem \ref{TH1}.

	\begin{lemma}\label{LeM}
		Assume (\ref{omega}) and let $BV_{rad}(\Omega)	=\{u\in BV(\Omega):\, u(x)=u(|x|)\}$. Then,  there exists $C> 0$ such that 
		\begin{equation} \label{BVinequaluty} 
			\sup_{x\in \overline{\Omega}}|u(x)|\leq Ca^{1-N}\|u\|. 
		\end{equation} 
		Hence, the embedding $BV_{rad}(\Omega) \hookrightarrow L^{\infty}(\Omega)$ is continuous and $BV_{rad}(\Omega) \hookrightarrow L^{p}(\Omega)$ is compact for all $p \in [1,\infty)$.
		
	\end{lemma}
	\begin{proof}
		From \cite[ Lemma 4.1 ]{GM1}, if $u\in  BV_{rad}(\mathbb{R}^{N})$ we have that  
		\begin{equation}\label{IN}
			|u(x)|\leq \frac{1}{|x|^{N-1}}\|u\|, \quad \mbox{a.e. in} \quad \mathbb{R}^N.
		\end{equation}
		Setting
		\begin{equation}
			\widetilde{u}=\left\{
			\begin{array}{l}
				u \;\;\text{if}\;\;x\in \Omega,\\
				0 \;\;\text{if}\;\;x\in \mathbb{R}^{N}\backslash \Omega,\\
			\end{array}\right.
		\end{equation} 
		in view of  \cite[Theorem 5.8 ]{Evn}, $\widetilde{u} \in BV(\mathbb{R}^{N})$ and 
		$$ 
		\int_{\mathbb{R}^{N}} |D\widetilde{u}|=\int_{\Omega} |Du|+\int_{\partial\Omega} |u|\,d\mathcal{H}^{N-1}.
		$$
		This combined with (\ref{IN}) shows (\ref{BVinequaluty}) and the continuous embedding $BV_{rad}(\Omega) \hookrightarrow L^{\infty}(\Omega)$. The compact embedding $BV_{rad}(\Omega) \hookrightarrow L^{p}(\Omega)$ for $p \in [1,\infty)$ follows combining the interpolation in the Lebesgue's space together with the compact embedding $BV_{rad}(\Omega) \hookrightarrow L^{1}(\Omega)$ and the continuous embedding $BV_{rad}(\Omega) \hookrightarrow L^{\infty}(\Omega)$.

	\end{proof}

	\section{Proof of Theorem \ref{TH1}}
	This section is devoted to prove Theorem \ref{TH1}. From now on, $p_{0}$ is the constant fixed in $(f_1)$. For each $p\in (1, p_{0})$, let us consider the following problem 
	\begin{equation}\label{Pp}\left\{
		\begin{array}{llc}
			u_{t}-\text{div}\left(|\nabla u|^{p-2}\nabla u\right)=f(u) & \text{in}\ & \Omega\times (0, +\infty) , \\
			u =0 & \text{in} & \partial\Omega\times (0,+\infty), \\
			u(x,0)=u_{0}(x)& \text{in} &\Omega , 
		\end{array}\right.
	\end{equation}
	where 
	\begin{equation*} \label{omegaab}
		\Omega=\{x\in \mathbb{R}^N\;:\;0<a<|x|<b\}.
	\end{equation*}
	In the sequel,  we denote by $E_p:W_{0,rad}^{1,p}(\Omega) \to \mathbb{R}$ the energy functional associated with problem (\ref{Pp}) given by 
	\begin{equation} \label{Ep}
		E_p(u)=\frac{1}{p}\int_{\Omega}|\nabla u|^{p}\,dx-\int_{\Omega}F(u)\,dx,
	\end{equation}
	and the Nehari set associated with $E_p$  given by
	$$
	\mathcal{N}_p=\left\{u \in W_{0,rad}^{1,p}(\Omega) \setminus \{0\}\,:\, E_p'(u)u=0\right\}.
	$$
	Hereafter, let us also denote by $d_p$ the following real number
	$$
	d_p=\inf_{u \in \mathcal{N}_p}E_p(u).
	$$
	The potential well associated with problem $\eqref{Pp}$ is the set
	\begin{equation}\label{pt7}
		W_p =\left\{u\in W_{0,rad}^{1,p}(\Omega)\,:\, E_p(u)< d_p \; \mbox{and} \; I_p(u)>0 \right\}\cup \{0\},
	\end{equation}
	where $I_p(u)=E_p'(u)u$ for all $u \in W_{0,rad}^{1,p}(\Omega)$.
	
	\vspace{0.5 cm}
	
	Our first lemma establishes an estimate from above for  $d_p$ that will be used later on. 
	
	\begin{lemma} \label{boundednessdp} There are $p_1 \in (1,p_0)$ and $M>0$ such that $d_p \leq M$ for all $p \in (1,p_1]$. 
		
	\end{lemma}
	\begin{proof} Let $\varphi \in C_{0,rad}^{\infty}(\Omega) \setminus \{0\}$. By $(f_1)$ and $(f_2)$, for each $p \in (1,p_0)$ there is $t_p >0$ such that $t_p \varphi \in \mathcal{N}_p$, that is, 
		$$
		t_p^{p-1}\int_{\Omega}|\nabla \varphi|^p\,dx=\int_{\Omega}f(t_p\varphi)\varphi\,dx.
		$$	 
		Since $\displaystyle \lim_{p \to 1^+}\int_{\Omega}|\nabla \varphi|^p\,dx=\int_{\Omega}|\nabla \varphi|\,dx$, the condition $(f_2)$ ensures that $(t_p)$ is bounded for $p$ close to 1. Now, using the inequality below
		$$
		d_p \leq E_p(t_p\varphi)\leq \frac{t_p^{p}}{p}\int_{\Omega}|\nabla \varphi|^p\,dx,
		$$ 
		we deduce that there are $p_1 \in (1,p_0)$ and $M>0$ such that $d_p \leq M$ for all $p \in (1,p_1)$. This proves the desired result.	
	\end{proof}

	From where it follows, decreasing $p_1 \in (1,p_0)$ if necessary, we can assume that  
	$$
	u_0 \in 	W_p =\left\{u\in W_{0,rad}^{1,p}(\Omega)\,:\, E_p(u)< d_p,\; I_p(u)>0 \right\}\cup \{0\} , \quad \forall p \in (1,p_1).
	$$
	In the sequel, we are supposing that $p \in (1,p_1)$ and $p_1 < \theta$.  
	
	For the reader's convenience we state the definition of weak solutions to (\ref{Pp}). 
	
	\begin{definition}(global weak solution)\label{DE}
		We say that $u\in L^{\infty}(0,+\infty; W_{0}^{1, p}(\Omega)) $ is a global weak solution of problem (\ref{Pp}) if, $u_{t} \in L^{2}(0, +\infty; L^{2}(\Omega))$ and the following equalities hold
		\begin{enumerate}
			\item $\int_{\Omega}u_{t}(t)v\,dx+\int_{\Omega} |\nabla u|^{p-2}\nabla u.\nabla v\,dx=\int_{\Omega} f(u(t))v\,dx$\\
			for each $v\in W_{0}^{1, p}(\Omega)\cap L^{2}(\Omega)$ and a.e time $0\leq t< +\infty$, and
			\item $u(0)=u_{0}$.
		\end{enumerate}
	\end{definition}
	
	In order to show the global existence of solutions to (\ref{Pp}) we will apply the Galerkin method to solve problem (\ref{Pp}). The proof will be divided into three steps : 
	\begin{description}
		\item[Step 1]  For $N\geq 1$,  we have the Gelfand triple
		$$ W^{1,p}_{0,rad}(\Omega) \hookrightarrow^{c,d} L_{rad}^{2}(\Omega)\hookrightarrow^{c,d} \left(W^{1,p}_{0,rad}(\Omega)\right)^{'},$$
		which is well defined because of (\ref{EMBW1p}). Here, $\hookrightarrow^{c,d} $ denotes a dense and continuous embedding. 
		
		Let $\{V_{m}\}_{m\in \mathbb{N}}$  be a Galerkin scheme of the separable  Banach space  $W_{0,rad}^{1,p}(\Omega)$, i.e,
		\begin{equation}
			V_{m}=span\{w_{1}, w_{2}, \ldots, w_{m}\},\;\; \overline{\bigcup_{m\in \mathbb{N}}V_{m}}^{\|\,.\|_{W_{0,rad}^{1, p}(\Omega)}}=W_{0,rad}^{1, p}(\Omega),
		\end{equation}
		where $\{w_{j}\}_{j=1}^{\infty}$ is an orthonormal basis in $L_{rad}^{2}(\Omega)$. As  $u_{0}\in W^{1,p}_{0,rad}(\Omega)$, there exists $u_{0m}\in V_{m}$ such that
		\begin{equation}\label{CVV}
			u_{m}(0)=u_{0m}\rightarrow u_{0}\;\;\text{strongly in }\; W^{1,p}_{0,rad}(\Omega) \;\text{as}\; m\rightarrow \infty.
		\end{equation}

		From now on, let us denote by $\|\cdot\|_{1,p}$ the usual norm in $W^{1,p}_{0,rad}(\Omega)$ given by 
		$$
		\|u\|_{1,p}=\|\nabla u\|_p, \quad \forall u \in W^{1,p}_{0,rad}(\Omega).
		$$

		For each $m$, we look for the approximate solutions $ u_{m}(x,t)=\sum\limits_{j=1}^{m}g_{jm}(t)w_{j}(x)$ satisfying the following identities :
		\begin{equation}\label{g1}
			\int_{\Omega} u_{mt}(t)w_{j}\,dx+\int_{ \Omega} |\nabla u_{m}(t)|^{p-2}\nabla u_{m}(t).\nabla w_{j}\,dx=\int_{\Omega} f(u_{m}(t))w_{j}\,dx,\;\;j\in \{1, \ldots, m\},
		\end{equation}
		with the initial conditions
		\begin{equation}\label{g2}
			u_{m}(0)=u_{0m}.
		\end{equation}
		Then $\eqref{g1}-\eqref{g2}$ is equivalent to the following initial value problem for a system of nonlinear ordinary differential equations on $g_{jm}$ :
		\begin{equation}\label{sys}
			\left\{
			\begin{array}{l}
				g'_{jm}(t)=H_{j}(g(t)),\;j=1,2, \ldots, m,\;t\in[0,t_{0}],\\
				g_{jm}(0)=a_{jm},\;j=1,2,\ldots, m.
			\end{array}\right.
		\end{equation}
		where $H_{j}(g(t))=-\int_{ \Omega} |\nabla u_{m}(t)|^{p-2}\nabla u_{m}(t).\nabla w_{j}\,dx+\int_{\Omega}f( u_{m})w_{j}\,dx$.
		By the Picard iteration method, there is $t_{0,m}> 0$ depending on $|a_{jm}|$ such that problem $\eqref{sys}$ admits a unique local solution $g_{jm}\in C^{1}([0,t_{0,m}])$. Hereafter, we will assume that $[0,T_{0,m})$ is the maximal interval of existence of the solution $u_m(t)$. 
		\item[Step 2] Multiplying the $j^{th}$ equation in $\eqref{g1}$ by $g'_{jm}(t)$ and summing over $j$ from $1$ to $m$, we obtain  
		\begin{equation}\label{g3}
			\|u_{mt}(t)\|^{2}_{2}+\frac{d}{dt}E_p(u_{m}(t))=0, \;\;t\in [0,T_{0,m}).
		\end{equation}
		Integrating  $\eqref{g3}$ over $(0, t)$ yields 
		\begin{equation}\label{EA}
			\int_{0}^{t}\|u_{ms}(s)\|^{2}_{2}\,ds+E_p(u_{m}(t))=E_{p}(u_{0m}), \; t\in [0, T_{0,m}).
		\end{equation}
		Since  $(u_{0m})$ converges to $u_{0}$ strongly in $W_{0,rad}^{1, p}(\Omega)$, by the continuity of $E$ it follows that 
		$$
		E_p(u_{0m})\rightarrow E_p(u_{0}),\;\text{as}\quad  m\rightarrow +\infty. 
		$$
		From the assumption that $E_p(u_{0})< d_p$,  we have $E_p(u_{0m})< d_p$ for sufficiently large $m$. This combined with $\eqref{EA}$ leads to
		\begin{equation}\label{vd}
			E_p(u_{m}(t))< d_p, \quad t\in [0,T_{0,m}), 
		\end{equation}
		for sufficiently large $m$. Now, we are going to show that $T_{0,m}=+\infty$ and 
		\begin{equation}\label{v1}
			u_{m}(t)\in W_p,\;\;\forall t\geq 0, 
		\end{equation}
		for sufficiently large $m.$ Suppose by contradiction that $u_{m}(t_1) \notin W_p$ for some $t_1 \in[0,t_{0,m})$. Let $t_{*} \in [0,T_{0,m})$ be the smallest time for which $u_{m}(t_{*})\notin W_p$. Then, by continuity of $u_{m}(t)$, we get $u_{m}(t_{*})\in \partial W_p$. Hence, it turns out that
		\begin{equation}\label{v4}
			E_p(u_{m}(t_{*}))=d_p,
		\end{equation} 
		or 
		\begin{equation}\label{v2}
			I_p(u_{m}(t_{*}))=0.
		\end{equation} 
		It is clear that $\eqref{v4}$ could not occur by $\eqref{vd}$ while if $\eqref{v2}$ holds then, by the definition of $d_p$,  
		$$ E_p(u_{m}(t_{*}))\geq \inf_{u\in \mathcal{N}}E_p(u)=d_{p}, $$
		which is also a contradiction with $\eqref{vd}$. Consequently, $\eqref{v1}$ holds.\\
		Combining $\eqref{vd}$ and $\eqref{v1}$, we find
		\begin{equation}
			\int_{0}^{t}\|u_{ms}(s)\|_{2}^{2}\,ds+\left(\frac{1}{p}-\frac{1}{\theta}\right) \int_{ \Omega} |\nabla u_{m}(t)|^p\,dx< d_p, \;t\in [0, T_{0,m}).
		\end{equation}
		Thus, it turns out that
		\begin{equation}\label{EST}
			\int_{ \Omega} |\nabla u_{m}(t)|^p\,dx< \frac{\theta pd_p}{\theta -p} \;\; \mbox{and} \;\; \int_{0}^{t}\|u_{ms}(s)\|^2_{2}\,ds< d_p,\;\;t\in [0, T_{0,m}),
		\end{equation}
		for $m$ large enough. Hence, the above estimates give $T_{0,m}=+\infty$. Here we are using the fact that if $T_{0,m}<+\infty$, then we must have $\displaystyle \lim_{t \to T_{0,m}^{-}}\|u_m(t)\|_{1,p}=+\infty$, see \cite[Lemma 2.4, p. 48]{JAM}. 
		
		An important inequality that we will be used later on is the following:
		\begin{claim}  \label{limitaumL1} $\|u_m(t)\|_2\leq \|u_0\|_{2}$ for all $t \geq 0$ and $m$ large enough.
		\end{claim}
		Indeed, multiplying the $j^{th}$ equation in $\eqref{g1}$ by $g_{jm}(t)$ and summing  up over $j=1, \ldots m$, we obtain
		\begin{equation*}
			\frac{1}{2}\frac{d}{dt}\|u_{m}(t)\|^{2}_{2}=-I_{p}(u_{m}(t))<0, \quad \forall t \in [0,+\infty),
		\end{equation*}
		for $m$ large. This implies that $t \mapsto \|u_m(t)\|_{2}^{2}$ is a decreasing function in $[0,+\infty)$. Thereby,  
		\begin{equation}\label{ET}
			\|u_{m}(t)\|^{2}_{2}\leq \|u_{0}\|^{2}_{2},\;\; \forall t\in [0, +\infty),
		\end{equation}
		for $m$ large enough, showing the claim.

		\item[Step 3]
		From $\eqref{EST}-\eqref{ET}$, there is a function $u$ and a subsequence of $(u_{m})$, still denoted by $(u_{m})$, such that
		\begin{equation}\label{LIm}
			\left\{\begin{array}{lcl}
				u_{m} \stackrel{*}{\rightharpoonup} u & \text{in}& L^{\infty}(0,T;W_{0,rad}^{1,p}(\Omega)),\\
				u_{mt}\rightharpoonup u_{t} &\text{in} & L^{2}(0,T;L_{rad}^{2}(\Omega)),\\
				u_{m}\rightharpoonup^{*} u &\text{in} & L^{\infty}(0,T;L_{rad}^{2}(\Omega)),\\
				-\text{div} \left(|\nabla u_{m}|^{p-2}\nabla u_{m}\right) \stackrel{*}{\rightharpoonup} \chi &\text{in} & L^{\infty}\left(0,T;\left(W_{0,rad}^{1, p}(\Omega)\right)'\right).\\
			\end{array}\right.
		\end{equation}
		Moreover, from  (\ref{EMBW1p}), (\ref{EST}) and \cite[Corollary 4, p. 85]{JS} for all $T> 0$ we have
		\begin{equation}\label{T1*}
			u_{m}\rightarrow u\;\;\text{in}\; C([0,T], C(\overline{\Omega})).
		\end{equation}
		In particular, 
		\begin{equation}\label{T1}
			u_{m}\rightarrow u\;\;\text{in}\; C([0,T], L^{\kappa}(\Omega)),  \;\forall \kappa\in [1, \infty],
		\end{equation}
		\begin{equation} \label{limitaul2}
			\|u(t)\|_2\leq \|u_0\|_{2}, \quad \forall t \geq 0, 
		\end{equation}
		
		and
		\begin{equation}\label{AL}
			u_{m}(x,t)\rightarrow u(x,t) \;\;\text{a.e.}\; x\in \Omega, \;\forall t\geq0.
		\end{equation}
		As $f$ is a continuous function, the limit (\ref{T1}) ensures that 
		\begin{equation}\label{T1**}
			f(u_{m}) \rightarrow f(u)\;\;\text{in}\; C([0,T], C(\overline{\Omega})).
		\end{equation}
		Thereby, for any fixed $j$, letting $m\rightarrow \infty$ in $\eqref{g1}$, we obtain
		\begin{eqnarray}\label{D1}
			\int_{0}^{t}\int_{\Omega} u_{\tau}(\tau)w_{j}\,dxd\tau+\int_{0}^{t}\langle \chi(\tau)), w_{j} \rangle\,d\tau=\int_{0}^{t}\int_{\Omega} f( u(\tau))w_{j}\,dxd\tau,\;\;\forall t\in [0, T].
		\end{eqnarray}
		From the density of $V_{m}$ in $W_{0,rad}^{1,p}(\Omega)$, it follows that for all $v\in W_{0,rad}^{1, p}(\Omega)$
		\begin{equation}\label{Sl}
			\int_{\Omega} u_{t}(t)v\,dx+\langle \chi (t),v\rangle =\int_{\Omega} f(u(t))v\,dx,\; \text{a.e.}\;t\in [0, T].
		\end{equation}
		By $\eqref{LIm}$ and \cite[see, Lemma 3.1.7]{Zh},
		$$ u_{m}(0)\rightarrow u(0)\;\text{weakly}\;\text{in}\; L^{2}(\Omega). $$
		However, by (\ref{CVV}) we know that $u_m(0) \to u_0$ in $W_{0,rad}^{1, p}(\Omega)$, in particular $u_m(0) \to u_0$ in $L^{2}(\Omega)$, and so, $u(0)=u_{0}$. This shows that $u$ satisfies the first initial condition. Next step is to prove that
		\begin{equation}\label{CL1}
			-\text{div} \left(|\nabla u|^{p-2}\nabla u\right)=\chi.
		\end{equation}
		In doing so, multiplying $\eqref{g1}$ by $g_{jm}(t)$ and summing up from $1$ to $m$, afterward integrating over $(0,T)$ yields
		\begin{equation}\label{ENR1}
			\int_{0}^{T}\langle -\text{div}\left(|\nabla u_{m}(t)|^{p-2} \nabla u_{m}(t)\right), u_{m}(t)\rangle\,dt=-\frac{1}{2}\|u_{m}(T)\|_{2}^{2}+\frac{1}{2}\|u_{m}(0)\|_{2}^{2}+\int_{0}^{T}\int_{\Omega} f( u_{m}(t))u_{m}(t)\,dxdt.
		\end{equation}
		From (\ref{T1*}) and (\ref{T1**}),  
		\begin{equation}\label{ERD}
			\int_{0}^{T}\int_{ \Omega} f(u_{m}(t))u_{m}(t)\,dxdt \rightarrow \int_{0}^{T}\int_{ \Omega} f(u(t))u(t)\,dxdt.	
		\end{equation}
		Letting $m\rightarrow \infty$ in $\eqref{ENR1}$, we obtain
		\begin{eqnarray*}\label{ENR2}
			\limsup_{m\rightarrow \infty}\int_{0}^{T}\langle -\text{div} \left(|\nabla u_{m}|^{p-2}\nabla u_{m}\right), u_{m}(t)\rangle\,dt&=& -\frac{1}{2}\liminf_{m\rightarrow \infty}\|u_{m}(T)\|^{2}_{2} +\frac{1}{2}\lim_{m\rightarrow \infty}\|u_{0m}\|_{2}^{2}\\
			&&+\lim_{m\rightarrow \infty}\int_{0}^{T}\int_{ \Omega}f(u_{m})u_{m}\,dxdt\\
			&\leq &-\frac{1}{2}\|u(T)\|^{2}_{2} +\frac{1}{2}\|u_{0}\|_{2}^{2}+\int_{0}^{T}\int_{ \Omega}f(u)u\,dxdt\\
			&&=\int_{0}^{T}\langle\chi (t), u(t)\rangle\,dt.
		\end{eqnarray*}
		Hence, from this and the theory of monotone operators (see \cite[ Remark 3.2.2]{Zh}), we conclude 
		\begin{equation}\label{Con}
			-\text{div} \left(|\nabla u|^{p-2}\nabla u\right)=\chi.
		\end{equation}
		Replacing $\eqref{Con}$ in $\eqref{Sl}$ yields
		\begin{equation}\label{SSl}
			\int_{\Omega} u_{t}(t)v\,dx+\int_{ \Omega} |\nabla u(t)|^{p-2}\nabla u(t).\nabla v\,dx =\int_{\Omega} f(u(t))v\,dx,\; \text{a.e. in}\;(0, T).
		\end{equation}
		As $T>0$ is arbitrary, it follows that 
		\begin{equation}\label{SSl*}
			\int_{\Omega} u_{t}(t)v\,dx+\int_{ \Omega} |\nabla u(t)|^{p-2}\nabla u(t).\nabla v\,dx =\int_{\Omega} f(u(t))v\,dx,\; \text{a.e. in}\; (0,+\infty),
		\end{equation}
		for all $v\in W_{0,rad}^{1, p}(\Omega)$.
		
		\begin{claim} \label{generalsolution}
			\begin{equation}\label{SSl**}
				\int_{\Omega} u_{t}(t)v\,dx+\int_{ \Omega} |\nabla u(t)|^{p-2}\nabla u(t).\nabla v\,dx =\int_{\Omega} f(u(t))v\,dx,\; \text{a.e. in}\; (0,+\infty),
			\end{equation}
			for all $v\in W_{0}^{1, p}(\Omega) \cap L^{2}(\Omega)$.	
		\end{claim}
		
		In order to prove Claim \ref{generalsolution}, we will use the Palais’ principle due to Squassina \cite[Theorem 4]{Squassina}. However, since the energy functional $E_p$ given in (\ref{Ep}) is not well defined in whole $v\in W_{0}^{1, p}(\Omega)$, we cannot use this principle directly in our problem. Here, we need to do the following trick: First of all, we fix $t >0$ such that equality in (\ref{SSl**}) is true. Setting $M=\|u(t)\|_\infty+1$, $g(x)=-u_t(t)(x)$ for all $x \in \Omega$ and  $h:\mathbb{R} \to \mathbb{R}$ by  
		$$
		h(t)=\left\{
		\begin{array}{l}
			f(M), \;\; \mbox{if} \;\; t \geq M, \\
			f(t), \;\; \mbox{if} \;\; |t| \leq M , \\
			f(-M), \;\; \mbox{if} \;\; t \leq -M,
		\end{array}
		\right.
		$$
		it follows that 
		\begin{equation}\label{SSl***}
			\int_{ \Omega} |\nabla u(t)|^{p-2}\nabla u(t).\nabla v\,dx =\int_{\Omega} h(u(t))v\,dx+\int_{\Omega}g(x)v\,dx,\; \forall v\in W_{0,rad}^{1, p}(\Omega). 
		\end{equation}
		Considering the functional $J:W_{0}^{1,p}(\Omega)\cap L^{2}(\Omega) \to \mathbb{R}$ given by 
		$$
		J(w)=\frac{1}{p}\int_{\Omega}|\nabla w|^{p}\,dx-\int_{\Omega}H(w)\,dx-\int_{\Omega}g(x)w(x)\,dx,
		$$
		where $H(t)=\int_{0}^{t}h(s)\,ds$ and the space $W_{0}^{1,p}(\Omega)\cap L^{2}(\Omega)$ endowed with its usual norm, that is, 
		$$
		\|u\|_{1,p,2}=\|\nabla u\|_p+\|u\|_2, \quad \forall u \in W_{0}^{1,p}(\Omega)\cap L^{2}(\Omega),
		$$
		which is a Banach space. A simple computation gives $J \in C^{1}(W_{0}^{1,p}(\Omega)\cap L^{2}(\Omega) ,\mathbb{R})$ and that $u(t)$ is a critical point of $J$ restricts to $W_{0,rad}^{1,p}(\Omega)$. Therefore, by Palais’ principle due to Squassina \cite{Squassina}, we deduce that $u(t)$ is a critical point of $J$ in whole $W_{0}^{1,p}(\Omega)\cap L^{2}(\Omega) $, that is, 
		\begin{equation}\label{SSl****}
			\int_{ \Omega} |\nabla u(t)|^{p-2}\nabla u(t).\nabla v\,dx =\int_{\Omega} h(u(t))v\,dx+\int_{\Omega}g(x)v\,dx,\; \forall v\in W_{0}^{1, p}(\Omega)\cap L^{2}(\Omega),
		\end{equation}
		or equivalently
		\begin{equation}\label{SSl*****}
			\int_{\Omega} u_{t}(t)v\,dx+\int_{ \Omega} |\nabla u(t)|^{p-2}\nabla u(t).\nabla v\,dx =\int_{\Omega} f(u(t))v\,dx,\; \forall v\in W_{0}^{1,p}(\Omega)\cap L^{2}(\Omega),
		\end{equation}
		showing the Claim \ref{generalsolution}.

		Next, we show that the solution $u$ satisfies the following energy inequality
		\begin{equation}\label{ESS}
			\int_{0}^{t}\|u_{\tau}(\tau)\|_{2}^{2}\,d\tau+E_{p}(u(t))\leq E_{p}(u_{0}), \;\text{a.e.}\; t\in [0,+\infty).
		\end{equation}
		To this end, let $\psi$ be a nonnegative function which belongs to $C_{0}([0,+\infty))$. From $\eqref{EA}$, 
		\begin{equation}\label{en2}
			\int_{0}^{T}\psi(t)\,dt\int_{0}^{T}\|u_{m\tau}(\tau)\|_{2}^{2}\,d\tau
			+\int_{0}^{T}E_{p}(u_{m}(t))\psi(t)\,dt=\int_{0}^{T}E(u_{m}(0))\psi(t)\,dt.
		\end{equation}
		The right-hand side of $\eqref{en2}$ converges to
		$$ \int_{0}^{T}E_{p}(u_{0})\psi(t)\,dt,$$
		as $m\rightarrow \infty$. The second term in the left-hand side $\int_{0}^{T}E_{p}(u_{m}(t))\psi(t)\,dt$ is lower semicontinuous with respect to the weak topology of $W_{0}^{1,p}(\Omega)$. Hence
		
		\begin{equation}\label{est}
			\int_{0}^{T}E_{p}(u(t))\psi(t)\,dt\leq \liminf_{m\rightarrow \infty} \int_{0}^{T}E_{p}(u_{m}(t))\psi(t)\,dt.
		\end{equation}
		Thus, the proof is now complete.
		
	\end{description}

	\subsection{Existence of solution for (\ref{P1})}
	
	In what follows, we set $p_m \to 1^+$ and $u_m=u_{p_m}$ the solution obtained in the last subsection, that is,  
	$$
	u_{m}\in  L^{\infty}(0,+\infty;W_{0,rad}^{1,p_m}(\Omega)),
	$$
	$$
	u_{mt} \in L^{2}(0,+\infty;L_{rad}^{2}(\Omega)),
	$$
	and
	\begin{equation}\label{SSl*2}
		\int_{\Omega} u_{mt}(t)v\,dx+\int_{ \Omega} |\nabla u_m(t)|^{p-2}\nabla u_m(t).\nabla v\,dx =\int_{\Omega} f(u_m(t))v\,dx,\; \text{a.e. in}\; (0,+\infty),
	\end{equation}
	for all $v\in W_{0}^{1, p_m}(\Omega) \cap L^{2}(\Omega)$ and $m \in \mathbb{N}$. Moreover, we also have 
	\begin{equation}\label{ESTZe}
		\int_{ \Omega} |\nabla u_{m}(t)|^{p_m}\,dx \leq \frac{\theta p_md_{p_m}}{\theta -p_m}, \;\; \;  \;\;\; \int_{0}^{t}\|u_{ms}(s)\|_{2}\,ds \leq  d_{p_m}\;\; \mbox{and} \quad \|u_{m}(t)\|_{2}\leq \|u_{0}\|_2, \; \mbox{for} \; t\in [0,+\infty).
	\end{equation}
	
	Since $\theta>1$, $p_m \to 1^+$ and $(d_{p_m})$ is bounded by Lemma \ref{boundednessdp}, there is $C_1>0$ such that 
	\begin{equation}\label{ESTZe2}
		\int_{ \Omega} |\nabla u_{m}(t)|^{p_m}\,dx< C_1, \;\; \; \int_{0}^{t}\|u_{ms}(s)\|^2_{2}\,ds< C_1 ,\;\;\; \|u_{m}(t)\|_{2}\leq \|u_{0}\|_2
		, \; \forall  \; t\in [0,\infty) \; \mbox{and} \; m \in \mathbb{N}.
	\end{equation}
	By Young inequality, 
	\begin{equation} \label{NEWEQUATIONZW} 
		\int_{\Omega}|\nabla u_m(t)|\,dx \leq \frac{1}{p_m} \int_{\Omega}|\nabla u_m(t)|^{p_m}\,dx +\frac{p_m-1}{p_m}|\Omega|,\; \forall t\in [0,+\infty) \; \mbox{and} \; m \in \mathbb{N}.
	\end{equation}
	Hence, there is $C_2>0$ such that 
	\begin{equation}\label{ESTZe1}
		\int_{ \Omega} |\nabla u_{m}(t)|\,dx \leq C_2,\; \forall t\in [0,\infty) \; \mbox{and} \; m \in \mathbb{N}. 
	\end{equation}
	
	Using  H\"older's inequality and (\ref{ESTZe2}), we get
	\begin{equation}\label{L1}
		\|u_{m}(t)\|_{1}\leq |\Omega|^{1/2}\|u_{m}\|_{2}\leq |\Omega|^{1/2}\|u_{0}\|_{2}, \; \mbox{for all } \;t \in [0,+\infty) \; \mbox{and} \; m \in\mathbb{N},
	\end{equation}
	showing that $(u_m)$ is a bounded sequence in $L^{\infty}(0,+\infty;L^{1}(\Omega))$. Recalling that the usual norm in $BV(\Omega)$ is 
	$$
	\|u\|_{BV(\Omega)}=\int_{\Omega}|Du|+\|u\|_{1}, \quad \forall u \in BV(\Omega),
	$$
	it follows from (\ref{ESTZe1}) and (\ref{L1}) that there is $C_4>0$ such that
	$$
	\int_{\Omega}|\nabla u_m|\,dx+\int_{\Omega}|u_m|\,dx\leq C_4, \quad \forall t \in (0,+\infty),
	$$
	showing that $(u_m)$ is bounded in $L^{\infty}(0,T;BV_{rad}(\Omega))$. Then, by Lemma \ref{LeM}, we derive that $(u_m)$ is also bounded in $L^{\infty}(0,T;L^{\infty}(\Omega))$. Moreover, this implies that  $(f(u_m))$ is a bounded sequence in $L^{\infty}(0,+\infty; L^{\infty}(\Omega))$. 
	
	As an immediate consequence of the above analysis, we deduce that
	\begin{equation}\label{LIm*}
		\left\{\begin{array}{lcl}
			u_{m} \stackrel{*}{\rightharpoonup} u & \text{in}& L^{\infty}(0,T;BV_{rad}(\Omega)),\\
			u_{mt}\rightharpoonup u_{t} &\text{in} & L^{2}(0,T;L_{rad}^{2}(\Omega)). \\
		\end{array}\right.
	\end{equation}
	
	By Lemma \ref{LeM} and \cite[Corollary 4, p. 85]{JS} for all $T> 0$ we have
	
	\begin{equation}\label{T1Ze}
		u_{m}\rightarrow u\;\;\text{in}\; C([0,T], L^{\kappa}(\Omega)),  \;\forall \kappa\in [1, \infty],
	\end{equation}
	and
	\begin{equation}\label{ALZe}
		u_{m}(x,t)\rightarrow u(x,t) \;\;\text{a.e.}\; x\in \Omega, \;\forall t\geq0.
	\end{equation}
	In particular $u_m(t) \to u(t)$ in $L^{1}(\Omega)$ for all $t \in [0,+\infty)$. From this, $u(t) \in BV(\Omega)$ for all $t \in [0,+\infty)$,  
	$$
	\liminf_{m \to +\infty}\int_{\Omega}\varphi| \nabla u_m(t)|\,dx \geq  \int_{\Omega}\varphi|Du(t)|\,dx, \quad \forall 0 \leq \varphi \in C_{0}^{1}(\Omega) \;\;\mbox{and} \;\; \forall  t \in [0,+\infty) \;\; (\mbox{see \cite{Ambrosio}}\;\; ).
	$$
	Moreover, from (\ref{secondnorm}) and (\ref{NEWEQUATIONZW}), 
	\begin{equation} \label{pmineqaulity}
		\liminf_{m \to +\infty}\int_{\Omega}| \nabla u_m(t)|^{p_m}\,dx \geq \liminf_{m \to +\infty}\left(\int_{\Omega}| \nabla u_m(t)|\,dx+\int_{\partial\Omega} |u_m|\, d\mathcal{H}^{N-1} \right) \geq \|u(t)\|, \quad \forall t \in [0, +\infty)
	\end{equation}
	and by (\ref{ALZe}),
	\begin{equation}\label{RS1Ze}
		f(u_{m}) \stackrel{*}{\rightharpoonup} f(u)\;\;\text{in}\; L^{\infty}(0, +\infty; L^{s}(\Omega)), \quad \forall s \in (1,\infty).
	\end{equation}
	
	\begin{claim} \label{newclaim} There are a vector field $z \in L^{\infty}(0,+\infty; L^{\infty}(\Omega))$ with $\mbox{div} \, z(t) \in L^{2}(\Omega)$  such that, up to subsequence,
		$$
		|\nabla u_m|^{p_n-2}\nabla u_n \rightharpoonup z \quad \mbox{in} \quad  L^{s}(0,T;L^{s}(\Omega)), \quad \forall s > 1 \quad \mbox{and} \quad \forall T>0, \leqno{(i)}
		$$	
		$$
		|z(t)|_\infty \leq 1,\quad \forall t>0, \leqno{(ii)}
		$$
		$$	
		(z(t),Du(t))=|Du(t)|, \quad \mbox{as measures on} \quad \Omega, \; \text{a.e. in}\; (0,+\infty), \leqno{(iii)}
		$$
		$$ [z(t), \nu]\in \text{sign}(-u(t))\;\;\mathcal{H}^{N-1}-\text{a.e}\;\;\text{on}\;\partial \Omega, \leqno{(iv)}
		$$
		and
		$$
		u_t(t)	-div \, z(t)=f(u(t)), \quad \mbox{in} \quad \mathcal{D}'(\Omega), \; \text{a.e. in}\; (0,+\infty). \leqno{(v)}
		$$
	\end{claim}	
	Let us prove this claim. For each $s>1$, there exists $m_0=m_0(s) \in \mathbb{N}$ such that $s(p_m-1)<p_m$ for all $m \geq m_0$. Thus, for each $T>0$, $|\nabla u_m|^{p_m-2}\nabla u_m  \in L^s(0,T;L^{s}(\Omega))$ for all $m \geq m_0$ and 
	$$
	\left(\int_{0}^{T}\||\nabla u_m(t)|^{p_m-2}\nabla u_m(t) \|^{s}_{L^{s}(\Omega)}\,dt\right)^{\frac{1}{s}} \leq |\Omega|^{\frac{1}{s}-\frac{p_m-1}{p_m}}\left(\int_{0}^{T}\|\nabla u_m\|_{p_m}^{(p_m-1)s}\,dt\right)^{\frac{1}{s}}, \quad \forall m \geq m_0.
	$$ 
	Thus, by (\ref{ESTZe2}),  
	\begin{equation} \label{NEWESTZ} 
		\left(\int_{0}^{T}\||\nabla u_m(t)|^{p_m-2}\nabla u_m(t) \|^{s}_{L^{s}(\Omega)}\,dt\right)^{\frac{1}{s}} \leq M^{\frac{(p_m-1)}{p_m}}|\Omega|^{\frac{1}{s}-\frac{p_m-1}{p_m}}T^{\frac{1}{s}}, \quad \forall m \geq n_0,
	\end{equation} 
	where $M=\displaystyle \sup_{m \in\mathbb{N}}\int_{\Omega}|\nabla v_m|^{p_m}\,dx$. Since $L^s(0,T;L^{s}(\Omega))$ is reflexive, there is $z \in L_{loc}^s(0,+\infty;L^{s}(\Omega))$, for all $s>1$, such that 
	\begin{equation} \label{limitez}
		|\nabla u_m|^{p_m-2}\nabla u_m  \rightharpoonup z \quad \mbox{in} \quad L^s(0,T;L^{s}(\Omega)), \quad \forall T>0.
	\end{equation} 
	The last limit combined with (\ref{NEWESTZ}) gives 
	$$
	|z|_{L^s(0,T;L^{s}(\Omega))} \leq (|\Omega|T)^{\frac{1}{s}}, \quad \forall s >1,
	$$
	from where it follows that $z \in  L^{\infty}(0,T;L^{\infty}(\Omega))$ with 
	$$
	|z|_{L^{\infty}(0,T;L^\infty(\Omega))}\leq 1, \quad \forall T>0.
	$$
	Hence, $z \in  L^{\infty}(0,+\infty;L^\infty(\Omega)))$ with $|z(t)|_{\infty}\leq 1$ for all $t >0$. 
	Finally the equality below
	$$
	\int_{\Omega}u_{mt}(t) \varphi\,dx+\int_{\Omega}|\nabla u_m(t)|^{p_m-2}\nabla u_m  \nabla \varphi \,dx =\int_{\Omega}f(u_m) \varphi\,dx, \quad \forall \varphi \in C_{0}^{1}(\Omega)
	$$
	leads to
	$$
	\int_{0}^{T}\int_{\Omega}u_{mt}(t) \varphi\,dxdt+\int_{0}^{T}\int_{\Omega}|\nabla u_m(t)|^{p_m-2}\nabla u_m  \nabla \varphi \,dxdt =\int_{0}^{T}\int_{\Omega}f(u_m) \varphi\,dxdt, 
	$$
	for all $\varphi \in C_{0}^{1}(\Omega)$ and $T>0.$ This together with the limits (\ref{limitez}) and (\ref{RS1Ze}) gives
	$$
	\int_{0}^{T}\int_{\Omega}u_{t}(t) \varphi\,dxdt+\int_{0}^{T}\int_{\Omega}z(t) \nabla \varphi \,dxdt =\int_{0}^{T}\int_{\Omega}f(u) \varphi\,dxdt, \quad \forall \varphi \in C_{0}^{1}(\Omega) \quad \mbox{and} \quad T>0,
	$$
	that is, 
	$$
	u_t(t)	-div \, z(t)=f(u(t)), \quad \mbox{in} \quad \mathcal{D}'(\Omega) \; \text{a.e. in}\; (0,+\infty),
	$$
	showing $(v)$, and that $\mbox{div} \, z(t) \in L^{2}(\Omega)$. Finally, in  order to prove $(iii)-(iv)$, we will adapt some arguments developed in \cite[page 57]{sergio}. As mentioned in \cite{sergio}, the item $(iii)$ follows if we prove the below inequality  
	$$
	-\int_{\Omega}u \mbox{div}z(t)\,dx -\int_{\Omega}uz(t)\nabla \varphi \,dx \geq \int_{\Omega}|Du(t)|\varphi, \quad \mbox{for all} \quad 0 \leq \varphi \in C_{0}^{1}(\Omega) \quad  \text{a.e. in}\; (0,+\infty). 
	$$
	From definition of $u_m$, it follows that for $0 \leq \varphi \in C_{0}^{1}(\Omega)$ and $T>0$, 
	$$
	\int_{0}^{T}\int_{\Omega}u_{mt}(t)u_m(t) \varphi\,dxdt+\int_{0}^{T}\int_{\Omega}|\nabla u_m(t)|^{p_m-2}\nabla u_m  \nabla( u_m\varphi) \,dxdt =\int_{0}^{T}\int_{\Omega}f(u_m) u_m\varphi\,dxdt,  
	$$
	and so, 
	$$
	\begin{array}{l}
		\displaystyle \int_{0}^{T}\int_{\Omega}u_{mt}(t)u_m(t) \varphi\,dxdt+\int_{0}^{T}\int_{\Omega}|\nabla u_m(t)|^{p_m}\varphi \,dxdt + \int_{0}^{T}\int_{\Omega}|\nabla u_m(t)|^{p_m-2}u_m(t)\nabla u_m(t)\nabla\varphi \,dxdt \\
		\mbox{} \\
		=\displaystyle \int_{0}^{T}\int_{\Omega}f(u_m(t)) u_m(t)\varphi\,dxdt.  	
	\end{array}
	$$
	A direct computation gives 
	$$
	\int_{0}^{T}\int_{\Omega}u_{mt}(t)u_m(t) \varphi\,dxdt \to \int_{0}^{T}\int_{\Omega}u_{t}(t)u(t) \varphi\,dxdt
	$$
	and 
	$$
	\int_{0}^{T}\int_{\Omega}f(u_m(t)) u_m(t)\varphi\,dxdt \to \int_{0}^{T}\int_{\Omega}f(u(t)) u(t)\varphi\,dxdt.
	$$
	Finally, we would like to point out that by lower semicontinuity,
	$$
	\liminf_{n \to +\infty}\int_{\Omega}|\nabla u_m(t)|^{p_m}\varphi\,dx \geq \liminf_{n \to +\infty}\int_{\Omega}\varphi|\nabla u_m(t)|\,dx\geq \int_{\Omega}\varphi|D u(t)| \quad \text{ in}\; (0,+\infty).
	$$
	The above analysis implies that
	$$
	-\int_{0}^{T}\int_{\Omega}u \mbox{div}z(t)\,dxdt -\int_{\Omega}uz(t)\nabla \varphi \,dxdt \geq \int_{0}^{T}\int_{\Omega}|Du|\varphi, \quad \forall T>0.
	$$
	Therefore,
	$$
	-\int_{\Omega}u \mbox{div}z(t)\,dx -\int_{\Omega}uz(t)\nabla \varphi \,dx \geq \int_{\Omega}|Du|\varphi, \quad \mbox{for all} \quad 0 \leq \varphi \in C_{0}^{1}(\Omega) \quad  \text{a.e. in}\; (0,+\infty),
	$$
	proving $(iii)$.
	
	Now, in order to prove $(iv)$, as showed in \cite{sergio}, it is enough to prove that
	$$
	\int_{\partial \Omega}\left(|u(t)|+u[z(t), \nu]\right)\,d\mathcal{H}^{N-1}\leq 0 \quad  \text{a.e. in}\; (0,+\infty).
	$$
	Using $(u_m-\varphi)$ as a test function, we get
	$$
	\begin{array}{l}
		\displaystyle \int_{0}^{T}\int_{\Omega}u_{mt}(t)(u_m(t)-\varphi)\,dxdt+\int_{0}^{T}\int_{\Omega}|\nabla u_m(t)|^{p_m} \,dxdt - \int_{0}^{T}\int_{\Omega}|\nabla u_m(t)|^{p_m-2}\nabla u_m(t)\nabla\varphi \,dxdt \\
		\mbox{} \\
		=\displaystyle \int_{0}^{T}\int_{\Omega}f(u_m(t)) (u_m(t)-\varphi)\,dxdt, \quad \forall T>0.  	
	\end{array}
	$$
	Letting $m \to +\infty$ and using (\ref{pmineqaulity}), we find
	$$
	\begin{array}{l}
		\displaystyle \int_{0}^{T}\int_{\Omega}|Du(t)|+\int_{0}^{T}\int_{\partial \Omega}|u(t)|\,d\mathcal{H}^{N-1} \leq -\int_{0}^{T}\int_{\Omega}u_t(t)(u-\varphi) \,dxdt+ \int_{0}^{T}\int_{\Omega}z(t)\nabla \varphi \,dxdt  \\
		\mbox{} \\
		- \displaystyle \int_{0}^{T}\int_{\Omega}f(u)\varphi\,dxdt+  \int_{0}^{T}\int_{\Omega}f(u)u\,dxdt,
	\end{array}
	$$
	that is,
	$$
	\begin{array}{l}
		\displaystyle \int_{0}^{T}\int_{\Omega}|Du(t)|\,dt+\int_{0}^{T}\int_{\partial \Omega}|u(t)|\,d\mathcal{H}^{N-1}dt \leq -\int_{0}^{T}\int_{\Omega}u_t(t)u(t)+ \displaystyle \int_{0}^{T}\int_{\Omega}f(u)u\,dxdt \\
		\mbox{}\\
		\hspace{7.3 cm} =-\displaystyle \int_{0}^{T}\int_{\Omega}\mbox{div}z(t)u(t)\,dxdt.
	\end{array}$$
	Then we have, by Green's formula (\ref{Gree}) 
	$$
	\displaystyle \int_{0}^{T}\int_{\Omega}|Du(t)|\,dt+\int_{0}^{T}\int_{\partial \Omega}|u(t)|\,d\mathcal{H}^{N-1} dt\leq -\int_{0}^{T}\int_{\partial\Omega}u(t)[z(t),\nu]\,d\mathcal{H}^{N-1}+\int_{0}^{T}\int_{\Omega}|Du(t)|\,dt
	$$
	which leads to
	$$
	\displaystyle \int_{0}^{T}\int_{\partial \Omega}\left(|u(t)|+u(t)[z(t),\nu]\right)\,d\mathcal{H}^{N-1}dt\leq 0, \quad \forall T>0.
	$$
	Therefore,
	$$
	\displaystyle \int_{\partial \Omega}(|u(t)|+u(t)[z(t),\nu])\,d\mathcal{H}^{N-1} \leq 0, \quad \mbox{a.e. in} \quad (0,+\infty),
	$$
	showing the desired result. In order to prove (\ref{ENRR1}) we can use a similar argument as in the proof of (\ref{en2}) and the weak lower semicontinuity of the total variation. Hence, the proof is now complete.
	
	\section{Proof of Theorem \ref{TH2}}
	In this section we are concerned with the proof Theorem \ref{TH2}. Here we just sketch the proof of this theorem since it proceeds as above.

	\begin{description}
		\item[Step 1] For $N\geq 2$,  we have the Gelfand triple
		$$ 
		W^{1,p}_{0}(\Omega)\cap L^{2}(\Omega)\hookrightarrow^{c,d} L^{2}(\Omega)\hookrightarrow^{c,d} \left(W^{1,p}_{0}(\Omega)\cap L^{2}(\Omega)\right)^{'}.
		$$
		In this section, we denote by $\|\,\,\,\|_{1,p,2}$ the usual norm in $W^{1,p}_{0}(\Omega)\cap L^{2}(\Omega)$ given by 
		$$
		\|u\|_{1,p,2}=\|\nabla u\|_{p}+\|u\|_2, \quad \forall u \in W^{1,p}_{0}(\Omega)\cap L^{2}(\Omega).
		$$

		Let $\{V_{m}\}_{m\in \mathbb{N}}$  be a Galerkin scheme of the separable  Banach space  $W_{0}^{1,p}(\Omega) \cap L^{2}(\Omega)$, i.e,
		\begin{equation}
			V_{m}=span\{w_{1}, w_{2}, \ldots, w_{m}\},\;\; \overline{\bigcup_{m\in \mathbb{N}}V_{m}}^{\|\,.\|_{1,p,2}}=W_{0}^{1,p}(\Omega) \cap L^{2}(\Omega),
		\end{equation}
		where $\{w_{j}\}_{j=1}^{\infty}$ is an orthonormal basis in $L^{2}(\Omega)$. Since $u_{0}\in W^{1,p}_{0}(\Omega)\cap L^{2}(\Omega)$ then we can find $u_{0m}\in V_{m}$ such that
		\begin{equation}\label{CVV1}
			u_{m}(0)=u_{0m}\rightarrow u_{0}\;\;\text{strongly in }\; W^{1,p}_{0}(\Omega)\cap L^{2}(\Omega) \;\text{as}\; m\rightarrow \infty.
		\end{equation}
		For each $m$, we look for the approximate solutions $ u_{m}(x,t)=\sum\limits_{j=1}^{m}g_{jm}(t)w_{j}(x)$ satisfying the following identities :
		\begin{equation}\label{g11}
			\int_{\Omega} u_{mt}(t)w_{j}\,dx+\int_{ \Omega} |\nabla u_{m}(t)|^{p-2}\nabla u_{m}(t).\nabla w_{j}\,dx=\int_{\Omega} f(u_{m}(t))w_{j}\,dx,\;\;j\in \{1, \ldots, m\},
		\end{equation}
		with the initial conditions
		\begin{equation}\label{g21}
			u_{m}(0)=u_{0m}.
		\end{equation}
		Then  $\eqref{g11}-\eqref{g21}$ is equivalent to the following initial value problem for a system of nonlinear ordinary differential equations on $g_{jm}$ :
		\begin{equation}\label{sys1}
			\left\{
			\begin{array}{l}
				g'_{jm}(t)=H_{j}(g(t)),\;j=1,2, \ldots, m,\;t\in[0,t_{0}],\\
				g_{jm}(0)=a_{jm},\;j=1,2,\ldots, m.
			\end{array}\right.
		\end{equation}
		where $H_{j}(g(t))=-\int_{ \Omega} |\nabla u_{m}(t)|^{p-2}\nabla u_{m}(t).\nabla w_{j}\,dx+\int_{\Omega}f( u_{m})w_{j}\,dx$.
		By the Picard iteration method, there is $t_{0,m}> 0$ depending on $|a_{jm}|$ such that problem $\eqref{sys}$ admits a unique local solution $g_{jm}\in C^{1}([0,t_{0,m}])$. Hereafter, we will assume that $[0,T_{0,m})$ is the maximal interval of existence of the solution $u_m(t)$. 
		\item[Step 2] Multiplying the $j^{th}$ equation in $\eqref{g11}$ by $g'_{jm}(t)$ and summing over $j$ from $1$ to $m$, we obtain  
		\begin{equation}\label{g31}
			\|u_{mt}(t)\|^{2}_{2}+\frac{d}{dt}E_p(u_{m}(t))=0, \;\;t\in [0,T_{0,m}).
		\end{equation}
		Integrating  $\eqref{g3}$ over $(0, t)$ yields 
		\begin{equation}\label{EA1}
			\int_{0}^{t}\|u_{mt}(s)\|^{2}_{2}\,ds+E_p(u_{m}(t))=E(u_{0m}), \; t\in [0, T_{0,m}).
		\end{equation}
		Since  $u_{0m}$ converges to $u_{0}$ strongly in $W_{0}^{1, p}(\Omega)\cap L^{2}(\Omega)$, by the continuity of $E$ it follows that 
		$$
		E_p(u_{0m})\rightarrow E_p(u_{0}),\;\text{as}\quad  m\rightarrow +\infty. 
		$$
		From the assumption that $E_p(u_{0})< d_p$,  we have $E_p(u_{0m})< d_p$ for sufficiently large $m$. This combined with $\eqref{EA}$ yields 
		\begin{equation}\label{vd1}
			E_p(u_{m}(t))< d_p, \quad t\in [0,T_{0,m}), 
		\end{equation}
		for sufficiently large $m$. Proceeding similarly to above, we immediately have 
		\begin{equation}\label{v11}
			u_{m}(t)\in W_p,\;\;\forall t\in [0,T_{0,m}).  
		\end{equation}
		Gathering $\eqref{vd}$ and $\eqref{v1}$ we deduce 
		\begin{equation}
			\int_{0}^{t}\|u_{ms}(s)\|_{2}^{2}\,ds+\left(\frac{1}{p}-\frac{1}{\theta}\right) \int_{ \Omega} |\nabla u_{m}(t)|^p\,dx< d_p, \;t\in [0, T_{0,m}).
		\end{equation}
		Thus, it turns out that
		\begin{equation}\label{EST1}
			\int_{ \Omega} |\nabla u_{m}(t)|^p\,dx< \frac{\theta pd_p}{\theta -p}, \;\; \; \int_{0}^{t}\|u_{ms}(s)\|^2_{2}\,ds< d_p,\;\;t\in [0, T_{0,m}).
		\end{equation}
		Now, multiplying the $j^{th}$ equation in $\eqref{g11}$ by $g_{jm}(t)$ and summing  up over $j=1, \ldots m$, we obtain
		\begin{equation*}
			\frac{1}{2}\frac{d}{dt}\|u_{m}(t)\|^{2}_{2}=-I_{p}(u_{m}(t))<0 \quad \forall t\in [0, T_{0,m}).
		\end{equation*}
		Hence, the function $t \mapsto \|u_m(t)\|_{2}^{2}$ for $t \in [0,T_{0,m})$ is decreasing, and so, 
		\begin{equation}\label{ET1}
			\|u_{m}(t)\|^{2}_{2}\leq \|u_{0}\|^{2}_{2},\;\;t\in [0, T_{0, m}),
		\end{equation}
		for $m$ large enough. The above estimates ensure that $T_{0,m}=+\infty$. Here we are using the fact that if $T_{0,m}<+\infty$, then we must have $\displaystyle \lim_{t \to T_{0,m}^{-}}\|u_m(t)\|_{1,p,2}=+\infty$, see \cite[Lemma 2.4, p. 48]{JAM}.

		\item[Step 3]
		From $\eqref{EST1}-\eqref{ET1}$ we get the existence of a function $u$ and a subsequence of $(u_{m})$ still denoted by $(u_{m})$ such that
		\begin{equation}\label{LIm1}
			\left\{\begin{array}{lcl}
				u_{m} \stackrel{*}{\rightharpoonup} u & \text{in}& L^{\infty}(0,T;W_{0}^{1,p}(\Omega)\cap L^{2}(\Omega)),\\
				u_{mt}\rightharpoonup u_{t} &\text{in} & L^{2}(0,T;L^{2}(\Omega)),\\
				-\text{div} \left(|\nabla u_{m}|^{p-2}\nabla u_{m}\right) \stackrel{*}{\rightharpoonup} \chi &\text{in} & L^{\infty}\left(0,T;\left(W_{0}^{1, p}(\Omega)\right)'\right).\\
			\end{array}\right.
		\end{equation}
		Moreover, from  $\eqref{EST1}$ and \cite[Corollary 4, p. 85]{JS} for all $T> 0$ we have
		
		\begin{equation}\label{T11}
			u_{m}\rightarrow u\;\;\text{in}\; C([0,T], L^{\kappa}(\Omega)),  \;\forall \kappa\in [1, p^{*})
		\end{equation}
		and
		\begin{equation}\label{AL1}
			u_{m}(x,t)\rightarrow u(x,t) \;\;\text{a.e.}\; x\in \Omega, \;\forall t\geq0.
		\end{equation}
		As $f$ is a continuous function, the limit (\ref{AL1}) yields that 
		\begin{equation}\label{T11**}
			f(u_{m}) \rightarrow f(u)\;\;\text{a.e.}\; x\in \Omega, \;\forall t\geq0.
		\end{equation}
		On the other hand, from $(f_{3})$, (\ref{ET1}) and the H\"older inequality,  
		\begin{equation}\label{K1}
			\int_{\Omega} |f(u_{m})|^{2}\,dx\leq 2C\left(|\Omega|+|\Omega|^{\frac{4-2q}{2}}\|u_{0}\|^{2q-2}_{2}\right)
		\end{equation}
		for $m$ large enough. Therefore, from (\ref{AL1}) and (\ref{K1}), 
		\begin{equation}\label{VC}
			f(u_{m})\stackrel{*}{\rightharpoonup} f(u)\;\;\;\text{in}\; L^{\infty}(0, T;L^{2}(\Omega)).
		\end{equation}
		For any fixed $j$, letting $m\rightarrow \infty$ in $\eqref{g11}$, we obtain
		\begin{eqnarray}\label{D11}
			\int_{0}^{t}\int_{\Omega} u_{\tau}(\tau)w_{j}\,dxd\tau+\int_{0}^{t}\langle \chi(\tau)), w_{j} \rangle\,d\tau=\int_{0}^{t}\int_{\Omega} f( u(\tau))w_{j}\,dxd\tau\;\;\forall t\in [0, T].
		\end{eqnarray}
		From the density of $V_{m}$ in $W_{0}^{1,p}(\Omega)\cap L^{2}(\Omega)$, it follows that for all $v\in W_{0}^{1, p}(\Omega)\cap L^{2}(\Omega)$
		\begin{equation}\label{Sl1}
			\int_{\Omega} u_{t}(t)v\,dx+\langle \chi (t),v\rangle =\int_{\Omega} f(u(t))v\,dx,\; \text{a.e.}\;t\in [0, T].
		\end{equation}
		By using a similar argument as in the proof of Theorem \ref{TH1}, one can shows that $u$ satisfies the first initial condition. Next step is to prove that
		\begin{equation}\label{CL11}
			-\text{div} \left(|\nabla u|^{p-2}\nabla u\right)=\chi.
		\end{equation}
		Since the proof of (\ref{CL11}) will be similar to that of Theorem \ref{TH1}, we only need to prove the following claim 
		\begin{claim}\label{CLM}
			$\int_{0}^{T}\int_{ \Omega} f(u_{m}(t))u_{m}(t)\,dxdt \rightarrow \int_{0}^{T}\int_{ \Omega} f(u(t))u(t)\,dxdt$ as $m\rightarrow \infty$. 
		\end{claim}  
		Indeed, from $(f_{3})$, (\ref{EST1}) and using H\"older's inequality for each measurable set $V\subset \Omega \times [0,T]$, we have 
		$$ \int_{V}f(u_{m})u_{m}\,dxdt\leq CT\left(3|V|^{1/2}\|u_{0}\|_{2}+|V|^{\frac{p^{*}}{p^{*}-q}}\left(\frac{\theta p d_{p}}{\theta-p}\right)^{q/p}\right)$$
		for $m$ large enough. In view of (\ref{AL1}), we have $f(u_{m}(x,t))u_{m}(x,t)\rightarrow f(u(x,t))u(x,t)$ a.e. in $\Omega \times [0, T]$. Thus, the desired result follows from the Vitali's convergence theorem.
		
		Replacing $\eqref{Con}$ in $\eqref{Sl1}$ yields
		\begin{equation}\label{SSl1}
			\int_{\Omega} u_{t}(t)v\,dx+\int_{ \Omega} |\nabla u(t)|^{p-2}\nabla u(t).\nabla v\,dx =\int_{\Omega} f(u(t))v\,dx,\; \text{a.e. in}\;(0, T), \;\;\forall  v\in W_{0}^{1, p}(\Omega)\cap L^{2}(\Omega).
		\end{equation}
		As $T>0$ is arbitrary, it follows that 
		\begin{equation}\label{SSl*1}
			\int_{\Omega} u_{t}(t)v\,dx+\int_{ \Omega} |\nabla u(t)|^{p-2}\nabla u(t).\nabla v\,dx =\int_{\Omega} f(u(t))v\,dx,\; \text{a.e. in}\; (0,+\infty),
		\end{equation}
		for all $v\in W_{0}^{1, p}(\Omega)\cap L^{2}(\Omega)$.
		
	\end{description}

	\subsection{Existence of solution for the original problem}
	
	In what follows, we set $p_m \to 1^+$ and $u_m=u_{p_m}$ the solution obtained in the last version, that is,  
	$$
	u_{m}\in  L^{\infty}(0,+\infty;W_{0}^{1,p_m}(\Omega)\cap L^{2}(\Omega)),
	$$
	$$
	u_{mt} \in L^{2}(0,+\infty;L^{2}(\Omega)),
	$$
	and 
	\begin{equation}\label{SSl*21}
		\int_{\Omega} u_{mt}(t)v\,dx+\int_{ \Omega} |\nabla u_m(t)|^{p-2}\nabla u_m(t).\nabla v\,dx =\int_{\Omega} f(u_m(t))v\,dx,\; \text{a.e. in}\; (0,+\infty),
	\end{equation}
	for all $v\in W_{0}^{1, p_m}(\Omega)\cap L^{2}(\Omega)$.
	
	Moreover, we also have 
	\begin{equation}\label{ESTZe11}
		\int_{ \Omega} |\nabla u_{m}(t)|^{p_m}\,dx \leq \frac{\theta p_md_{p_m}}{\theta -p_m}, \;\; \; \int_{0}^{t}\|u_{ms}(s)\|^2_{2}\,ds \leq  d_{p_m},\;\;\;\|u_{m}(t)\|_{2}\leq \|u_{0}\|
		,\;\; t\in [0,+\infty).
	\end{equation}
	Since $\theta>1$, $p_m \to 1^+$ and $(d_{p_m})$ is a bounded sequence by Lemma \ref{boundednessdp}, there is $C_1>0$ such that 
	\begin{equation}\label{ESTZe21}
		\int_{ \Omega} |\nabla u_{m}(t)|^{p_m}\,dx< C_1, \;\; \; \int_{0}^{t}\|u_{ms}(s)\|^2_{2}\,ds< C_1 ,\;\;\;\|u_{m}(t)\|_{2}\leq \|u_{0}\|
		,\;\;, \quad \forall m \in \mathbb{N}.
	\end{equation}
	By Young inequality, 
	$$
	\int_{\Omega}|\nabla u_m(t)|\,dx \leq \frac{1}{p_m} \int_{\Omega}|\nabla u_m(t)|^{p_m}\,dx +\frac{p_m-1}{p_m}|\Omega|, \;\;t\in [0,\infty), \quad \forall m \in \mathbb{N}.
	$$
	Hence, there is $C_2>0$ such that 
	\begin{equation}\label{ESTZe111}
		\int_{ \Omega} |\nabla u_{m}(t)|\,dx< C_2, \;\;t\in [0,+\infty), \quad \forall m \in \mathbb{N}. 
	\end{equation}

	\begin{claim} \label{limitationZ1}  $(u_m)$ is a bounded sequence in $L^{\infty}(0,+\infty;L^{1}(\Omega))$. Hence, $(u_m)$ is a bounded sequence in $L^{\infty}(0,+\infty;BV(\Omega))$ and $(f(u_m))$ is a bounded sequence in $L^{\infty}(0,+\infty; L^{2}(\Omega))$.
	\end{claim}
	Indeed, from  H\"older's inequality and (\ref{ESTZe21}) we get
	\begin{equation}\label{L11}
		\|u_{m}(t)\|_{1}\leq |\Omega|^{1/2}\|u_{m}\|^{2}_{2} \leq |\Omega|^{1/2}\|u_{0}\|^{2}_{2},\quad \forall t \in (0,+\infty) \quad \mbox{and} \quad \forall m \in\mathbb{N}
	\end{equation}
	showing that $(u_m)$ is a bounded sequence in $L^{\infty}(0,+\infty;L^{1}(\Omega))$. Recalling that the usual norm in $BV(\Omega)$ is 
	$$
	\|u\|=\int_{\Omega}|Du|+|u|_{1}, \quad \forall u \in BV(\Omega),
	$$
	it follows from (\ref{ESTZe11}) and (\ref{L11}) that there is $C_4>0$ such that
	$$
	\int_{\Omega}|\nabla u_m|\,dx+\int_{\Omega}|u_m|\,dx\leq C_4, \quad \forall t \in (0,+\infty)
	$$
	showing that $(u_m)$ is bounded in $L^{\infty}(0,T;BV(\Omega))$.
	Since $2q-2<2$, the second part of the claim  follows directly from (\ref{K1}). Hence, $(f(u_m))$ is a bounded sequence in $L^{\infty}(0,+\infty; L^{2}(\Omega))$.
	
	The last claim permits to deduce that
	\begin{equation}\label{LIm*1}
		\left\{\begin{array}{lcl}
			u_{m} \stackrel{*}{\rightharpoonup} u & \text{in}& L^{\infty}(0,T;BV(\Omega)\cap L^{2}(\Omega)),\\
			u_{mt}\rightharpoonup u_{t} &\text{in} & L^{2}(0,T;L^{2}(\Omega)),\\
		\end{array}\right.
	\end{equation}
	
	By \cite[Corollary 4, p. 85]{JS} for all $T> 0$ we have
	
	\begin{equation}\label{T1Ze1}
		u_{m}\rightarrow u\;\;\text{in}\; C([0,T], L^{\kappa}(\Omega)),  \;\forall \kappa\in [1, 1^{*}).
	\end{equation}
	and
	\begin{equation}\label{ALZe1}
		u_{m}(x,t)\rightarrow u(x,t) \;\;\text{a.e.}\; x\in \Omega, \;\forall t\geq0.
	\end{equation}
	In particular $u_m(t) \to u(t)$ in $L^{1}(\Omega)$ for all $t \in [0,+\infty)$. From this, we have that $u(t) \in BV(\Omega)$ for all $t \in [0,+\infty)$ and 
	$$
	\liminf_{m \to +\infty}\int_{\Omega}\varphi| \nabla u_m(t)|\,dx \geq  \int_{\Omega}\varphi|Du(t)|\,dx, \quad \forall 0 \leq \varphi \in C_{0}^{1}(\Omega) \;\;\mbox{and} \;\;  \;\; t \in [0,+\infty). \;\; (\mbox{see \cite{Ambrosio}}\;\; )
	$$
	and
	$$
	\liminf_{m \to +\infty}\int_{\Omega}| \nabla u_m(t)|^{p_m}\,dx \geq \liminf_{m \to +\infty}\left(\int_{\Omega}| \nabla u_m(t)|\,dx+\int_{\partial\Omega} |u_m|\, d\mathcal{H}^{N-1} \right) \geq \|u(t)\|, \quad \forall t \in [0, +\infty)
	$$
	Moreover, from (\ref{T1Ze1}),
	\begin{equation}\label{RS1Ze1}
		f(u_{m}) \stackrel{*}{\rightharpoonup} f(u)\;\;\text{in}\; L^{\infty}(0, +\infty; L^{2}(\Omega)).
	\end{equation}
	
	\begin{claim} \label{newclaim1} There are a vector field $z \in L^{\infty}(0,+\infty; L^{\infty}(\Omega))$ with $\mbox{div} \, z(t) \in L^{2}(\Omega)$  such that, up to subsequence,
		$$
		|\nabla u_m|^{p_m-2}\nabla u_m \rightharpoonup z \quad \mbox{in} \quad  L^{s}(0,T;L^{s}(\Omega)), \quad \forall s > 1 \quad \mbox{and} \quad \forall T>0, \leqno{(i)}
		$$	
		$$
		|z(t)|_\infty \leq 1,\quad \forall t>0, \leqno{(ii)}
		$$
		$$	
		(z(t),Du(t))=|Du(t)|, \quad \mbox{as measures on} \quad \Omega, \; \text{a.e. in}\; (0,+\infty), \leqno{(iii)}
		$$
		$$ [z(t), \nu]\in \text{sign}(-u(t))\;\;\mathcal{H}^{N-1}-\text{a.e}\;\;\text{on}\;\partial \Omega, \leqno{(iv)}
		$$
		and
		$$
		u_t(t)	-div \, z(t)=f(u(t)), \quad \mbox{in} \quad \mathcal{D}'(\Omega), \; \text{a.e. in}\; (0,+\infty). \leqno{(v)}
		$$
	\end{claim}	
	Now the proof of this claim follows as in the proof of Claim \ref{newclaim}. Therefore the proof of Theorem \ref{TH2} is now complete.

\end{document}